\documentclass[12pt, reqno]{amsart}
\usepackage{amsmath, amsthm, amscd, amsfonts, amssymb, graphicx, color}
\usepackage[bookmarksnumbered, colorlinks, plainpages]{hyperref}
\hypersetup{colorlinks=true,linkcolor=red, anchorcolor=green, citecolor=cyan, urlcolor=red, filecolor=magenta, pdftoolbar=true}
\input{mathrsfs.sty}
\usepackage{xcolor}
\usepackage{tikz}
\usetikzlibrary{decorations.pathreplacing,angles,quotes}
\definecolor{cof}{RGB}{219,144,71}
\definecolor{pur}{RGB}{186,146,162}
\definecolor{greeo}{RGB}{91,173,69}
\definecolor{greet}{RGB}{52,111,72}
\usepackage{xcolor}
\usepackage{ulem}
\newtheorem{theorem}{Theorem}[section]
\newtheorem{lemma}[theorem]{Lemma}
\newtheorem{proposition}[theorem]{Proposition}
\newtheorem{cor}[theorem]{Corollary}
\theoremstyle{definition}
\newtheorem{definition}[theorem]{Definition}
\newtheorem{example}[theorem]{Example}

\theoremstyle{remark}
\newtheorem{remark}[theorem]{\bf{Remark}}
\numberwithin{equation}{section}
\newcommand{\sm}{{\rm Sm}\,}

\newcommand{\rint}{{\rm Int_r}\,}
\newcommand{\ext}{{\rm Ext}\,}

\newcommand{\aff}{\rm aff}
\newcommand{\co}{\rm co}
\newcommand{\spn}{{\rm span}}
\allowdisplaybreaks
\begin{document}
	\title[ An improvement of the Blanco-Koldobsky-Turn\v{s}ek characterization]{An improvement of the Blanco-Koldobsky-Turn\v{s}ek characterization of isometries}

	\author[Manna, Mandal, Paul and Sain  ]{Jayanta Manna, Kalidas Mandal,  Kallol Paul and Debmalya Sain }

	\address[Manna]{Department of Mathematics\\ Jadavpur University\\ Kolkata 700032\\ West Bengal\\ INDIA}
	\email{iamjayantamanna1@gmail.com}

	\address[Mandal]{Department of Mathematics\\ Jadavpur University\\ Kolkata 700032\\ West Bengal\\ INDIA}
	\email{kalidas.mandal14@gmail.com}

	\address[Paul]{Vice-Chancellor, Kalyani University \& Professor of Mathematics\\ Jadavpur University (on lien) \\ Kolkata \\ West Bengal\\ INDIA}
	\email{kalloldada@gmail.com}
	
	\address[Sain]{Department of Mathematics\\ Indian Institute of Information Technology, Raichur\\ Karnataka 584135 \\INDIA}
	\email{saindebmalya@gmail.com}

	\newcommand{\acr}{\newline\indent}

	
	\begin{abstract}
		We present an improvement of the Blanco-Koldobsky-Turn\v{s}ek characterization of isometries in normed linear spaces by using the concept of level vectors of an operator. In this context, we  characterize level vectors entirely through directional preservation of Birkhoff-James orthogonality and analyze the associated geometric and structural phenomena that they induce. Furthermore, in spaces whose unit balls possess the \textit{Krein-Milman property}, we derive an additional refinement of the Blanco-Koldobsky-Turn\v{s}ek characterization of isometries.
		
	\end{abstract}
	
	\maketitle
	\section{Introduction}
	A cornerstone of the isometric theory of normed linear spaces is the following fundamental result established by Koldobsky, Blanco and Turn\v{s}ek \cite{BT06, K93}:\\
	
	\textit{A bounded linear operator on a normed linear space preserves Birkhoff-James orthogonality if and only if it is a scalar multiple of an isometry.}\\
	
	A local perspective on the preservation of Birkhoff-James orthogonality was developed in \cite{SMP24}, yielding refinements of this characterization for real Hilbert spaces and for certain real polyhedral spaces, including $\ell_{\infty}^{n}$ and $\ell_{1}^{n}$. Subsequent work in \cite{MMPS25} provided a finer characterization of isometries  for those spaces in terms of \textit{directional preservation of Birkhoff-James orthogonality}. More recently, in \cite{MMPS025}, further improvements were obtained for real normed linear spaces whose smooth points form a dense $G_\delta$ set. The purpose of this article is to obtain a proper refinement of this characterization for general normed linear spaces. To accomplish this, using the recently introduced concept of level vectors \cite{RST24, SRT21}, we establish in Theorem \ref{iso}:\\
	
	\textit{A bounded linear operator on a normed linear space has every nonzero element of the space as a level vector if and only if the operator is an isometry up to a scalar multiple.}\\
	
	Moreover, we show that if a bounded linear operator preserves Birkhoff-James orthogonality at a  nonzero point, then that point is necessarily a level vector of the operator. An illustrative example is provided to demonstrate that the converse implication need not hold. Therefore, our result provides a proper refinement of the Blanco-Koldobsky-Turn\v{s}ek characterization.  We now introduce the relevant terminology and notation used throughout this article.  \\
	
	We use the   symbols $ \mathbb{X}, \mathbb{Y} $ to denote normed linear spaces over the base field $\mathbb{K}=\mathbb{R}\text{ or }\mathbb{C}.$ The dual space of $\mathbb{X}$ is denoted by $\mathbb{X}^*.$ Let $B_{\mathbb{X}}$ and $S_{\mathbb{X}}$ be the unit ball and the unit sphere of $\mathbb{X}$, respectively.  The convex hull of a nonempty set $D\subset \mathbb{X}$ is denoted by $\co(D).$
	Let us recall that an affine space is defined as a translation of a subspace of the normed linear space. For a  nonempty set $D \subset \mathbb{X},$ the intersection of all affine spaces containing $D$ is denoted by $\aff(D).$ The relative interior of a nonempty set $D \subset \mathbb{X}$ is defined by $\text{Int}_r~D=\{x\in D : \text{ there exists }\epsilon>0\text{ such that } B(x,\epsilon)\cap \aff(D)\subseteq D\}.$ 
	For a nonempty convex set $D,$ let $\ext D$ be the collection of all extreme points of $D.$ 
	A convex set $F\subset S_{\mathbb{X}}$ is said to be a face of $B_{\mathbb{X}}$ if for any $u,v\in S_{\mathbb{X}},~ (1-t)u+tv\in F$ implies that $u,v\in F,$ where $0<t<1.$    For any nonzero $ x\in \mathbb{X},$ we define  $J(x)$  as the collection of all supporting functionals at $x,$ i.e., $ J(x) = \{ f \in S_{\mathbb{X}^*} : f(x) = \| x \| \}. $   A nonzero  $ x \in \mathbb{X} $ is a smooth point  if $J(x)$ is singleton. For any normed linear space the norm $\|\cdot\|$ is said  to be G\^ateaux differentiable at a nonzero $ x \in \mathbb{X} $ if for every  $ y \in \mathbb{X}, $ $ \lim\limits_{t \to 0}\frac{ \| x + ty \| - \| x \|}{t} $ exists finitely. Note that,  the norm $\|\cdot\|$ is G\^ateaux differentiable at $ x \in \mathbb{X} $ if and only if $ J(x)$ is singleton. Moreover, if $x$ is smooth then $ \lim\limits_{t \to 0}\frac{ \| x + ty \| - \| x \|}{t}= f_x(y),$ where $ f_x\in J(x).$ Let $\sm \mathbb{X}$ denote the collection of all smooth points of $\mathbb{X}.$\\
	
	Given  $u, v \in \mathbb{X},$  $u$ is Birkhoff-James orthogonal \cite{B35, J47}  to $v,$ written as $ u \perp_B v,$ if $ \| u + \lambda v\| \geq \|u\|$ for all scalars $\lambda. $ According to the  Hahn-Banach theorem,  for a given $u \in \mathbb{X},$ there exist enough $v \in \mathbb{X}$ such that $ u \perp_B v.$ Let $ u^{\perp_B} $ be the collection of all  $v\in \mathbb{X}$ such that $u \perp_B v.$ It is straightforward to observe that the Birkhoff-James orthogonality relation is homogeneous, i.e., $u\perp_B v\implies \alpha u\perp_B \beta v$ for all scalars $\alpha, \beta.$ In the case of an inner product space, the Birkhoff-James orthogonality  coincides with the usual inner product orthogonality.  For more on Birkhoff-James orthogonality,  readers can consult the recently published book \cite{Book24}.\\
	
	The set of all  bounded linear operators  from $\mathbb{X}$ to $\mathbb{Y}$ is denoted by $  \mathbb{L}(\mathbb{X}, \mathbb{Y}).$   Whenever $ \mathbb{X} = \mathbb{Y}, $ we simply write $ \mathbb{L}(\mathbb{X}, \mathbb{Y}) = \mathbb{L}(\mathbb{X}). $ For any $ T \in \mathbb{L}(\mathbb{X}, \mathbb{Y}) ,$ the adjoint of $T$ is denoted by $T^\times.$  We recall that  $  T \in \mathbb{L}(\mathbb{X}, \mathbb{Y}) $ is said to preserve  Birkhoff-James orthogonality if for  $ x, y \in \mathbb{X},$ $ x \perp_B y \Rightarrow Tx \perp_B Ty. $
	Quite naturally, for the local version of this property, the operator $T$ is said to preserve Birkhoff-James orthogonality at $ x \in \mathbb{X} $ if for all $y \in \mathbb{X},$ $ x \perp_B y \Rightarrow Tx \perp_B Ty $ and for the directional version of this property, the  operator $T$ is said to preserve Birkhoff-James orthogonality at $x\in\mathbb{X}$ with respect to a subspace $\mathbb{Z}$ contained in $x^{\perp_B}$ if $Tx\perp_B Tz,$ for all $z\in\mathbb{Z}.$  Several recent studies \cite{MMPS25,S20,S18,SMP24,SRT21} have demonstrated the importance of the local preservation of Birkhoff-James orthogonality in understanding the geometric properties of the underlying normed linear  spaces.
	
	We now state the definition of semi-inner products in a vector space \cite{G67,L61}, which plays a central role in our framework.
	\begin{definition}
		A function $[\cdot, \cdot] : \mathbb{X} \times \mathbb{X} \to \mathbb{K}$ is a semi-inner product if for any $\alpha, \beta \in \mathbb{K}$ and for any $x, y, z \in \mathbb{X}$, it satisfies the following:
		\begin{enumerate}
			\item $[\alpha x + \beta y, z] = \alpha [x, z] + \beta [y, z],$
			\item $[x, x] > 0,$ whenever $x \neq 0,$
			\item $|[x, y]|^2 \leq [x, x][y, y],$
			\item $[x, \alpha y] = \overline{\alpha} [x, y].$
		\end{enumerate}
	\end{definition}
	A semi-inner product space is a normed linear space $\mathbb X$ equipped with a semi-inner product $ [\cdot, \cdot]_{\mathbb{X}}$ such that the norm of any $x \in \mathbb{X}$ is given by $\|x\|= [x, x]_{\mathbb{X}}^{\frac 12}.$ Whenever we speak of an semi-inner product $[\cdot, \cdot]_{\mathbb X}$ in the context of a normed linear space $\mathbb{X}$, we implicitly assume that $[\cdot, \cdot]_{\mathbb X}$ is compatible with the norm. In fact, every normed linear space can be viewed as a semi-inner product space. It is worth noting that semi-inner products in normed linear spaces can be constructed by using the concept of duality maps (see \cite{C09}).\\
	
	We now recall the definition of level numbers from \cite{SRT21} for the reader’s convenience, as it plays an important role in our present study.
	\begin{definition}
		Let $T\in \mathbb{L}(\mathbb{X}, \mathbb{Y}).$	 A scalar $ k$ is called a level number of $T$, if there exists a nonzero $ x \in \mathbb{X} $, and semi-inner products $ [\cdot, \cdot]_\mathbb{X} $ and $[\cdot ,\cdot]_\mathbb{Y}$ in $\mathbb{X}$ and $\mathbb{Y}$, respectively, such that
		\[
		[Ty, Tx]_\mathbb{Y} = k [y, x]_\mathbb{X}~\text{ for all } y \in \mathbb{X}.
		\]
		Moreover, the vector $ x \in \mathbb{X}$ is called a level vector of $T$ corresponding to the level number $k.$ 
	\end{definition}
	It is worth noting that, for operators on inner product spaces, the level numbers and level vectors correspond exactly to the squares of the singular values of the operator and their associated right singular vectors. Additionally, in an arbitrary normed linear space, if $k$ is a level number of $T$ 
	associated with a level vector $x$, then $ k = \frac{\|Tx\|^{2}}{\|x\|^{2}}.$\\

	The main  objective of this article is to obtain a proper refinement of the famous Blanco-Koldobsky-Turn\v{s}ek characterization  of  isometries in normed linear spaces. To this end, we study level vectors of bounded linear operators between normed linear spaces. We first establish a complete characterization of level vectors and then investigate the structural and geometric consequences of the level vectors. Furthermore, we establish a connection between the level numbers of a bounded linear operator and those of its adjoint in the framework of reflexive Banach spaces.  Finally, we  refine the Blanco-Koldobsky-Turn\v{s}ek characterization  by  proving that every nonzero element of the domain of a bounded linear operator is a level vector if and only if the operator is an isometry up to a scalar multiple. We also derive an additional refinement of the Blanco-Koldobsky-Turn\v{s}ek characterization of isometries for those  spaces whose unit balls possess the \textit{Krein-Milman property}.

	\section{Main Results}
	Let us begin this section with a characterization of level vectors of a bounded linear operator between normed linear spaces, using the concept of directional preservation of Birkhoff-James orthogonality. It is worth noting that in \cite[Th. 2.7]{SRT21}, the authors established an analogous result for smooth normed linear spaces. Here we extend that result to arbitrary normed linear spaces.
	\begin{proposition}\label{level char BJ}
		Let  $T\in \mathbb{L}(\mathbb{X}, \mathbb{Y}).$  Then  a nonzero $x\in\mathbb{X}$ is a level vector of $T$  if and only if $T$ preserves Birkhoff-James orthogonality at $x$ with respect to $\ker f $ for some $f\in J(x).$
	\end{proposition}
	
	\begin{proof}
		Let $x(\neq 0)\in \mathbb{X}.$ Let $T$ preserve Birkhoff-James orthogonality at $x$ with respect to $\ker f $ for some $f\in J(x).$ Then there exists a semi-inner product $[\cdot, \cdot]_{\mathbb{X}_1}$ in $\mathbb{X}$ such that 
		\[[y,x]_{\mathbb{X}_1}=\|x\|f(y)~\text{ for all } y\in \mathbb{X}.\]
		Now, $\mathbb{X}=\spn \{x\} \oplus \ker f.$ Let $y\in \mathbb{X}$ be arbitrary.  Then there exist $h \in \ker f$ and $\alpha\in \mathbb{K}$ such that $y=\alpha x + h.$ Clearly, $x\perp_Bh.$ Since, $T$ preserves Birkhoff-James orthogonality at $x,$ it follows that $Tx\perp_B Th.$ Then there exists a semi-inner product $[\cdot, \cdot]_{\mathbb{Y}_1}$ in $\mathbb{Y}$ such that $[Th,Tx]_{\mathbb{Y}_1}=0.$
		Now,
		\begin{align*}
			[Ty,Tx]_{\mathbb{Y}_1}&= [T(\alpha x + h),Tx]_{\mathbb{Y}_1}\\
			&=\alpha \|Tx\|^2\\
			&= \frac{\|Tx\|^2}{\|x\|^2} ([\alpha x, x]_{\mathbb{X}_1}+[h,x]_{\mathbb{X}_1})\\
			&= \frac{\|Tx\|^2}{\|x\|^2} [y,x]_{\mathbb{X}_1}.
		\end{align*}
		So $x$ is a level vector of $T.$
		
		Conversely, let $x$ be a level vector of $T.$ Then there exist a semi-inner product $[\cdot, \cdot]_{\mathbb{X}_2}$ in $\mathbb{X}$ and  a semi-inner product $[\cdot, \cdot]_{\mathbb{Y}_2}$ in $\mathbb{Y}$ such that
		\[[Ty,Tx]_{\mathbb{Y}_2}=\frac{\|Tx\|^2}{\|x\|^2}[y,x]_{\mathbb{X}_2} \text{ for all } y\in \mathbb{X}.\]
		Let $f\in J(x)$ such that $[y,x]_{\mathbb{X}_2}=\|x\|f(y)~\text{ for all } y\in \mathbb{X}.$
		Let $z\in \ker f$ and so $x\perp_B z.$ Then
		\[[Tz,Tx]_{\mathbb{Y}_2}=\frac{\|Tx\|^2}{\|x\|^2}[z,x]_{\mathbb{X}_2}=\frac{\|Tx\|^2}{\|x\|}f(z)=0.\] 
		This implies that $Tx \perp_B Tz.$ Thus, $T$ preserves Birkhoff-James orthogonality at $x$ with respect to $\ker f.$
	\end{proof}
	
	The following example shows that an operator on a normed linear space need not preserve Birkhoff-James orthogonality at all of its level vectors.
	\begin{example}
		Let $\mathbb{X}=\ell_1^3(\mathbb R)$ and consider the operator $T\in \mathbb{L}(\mathbb{X})$ defined by 
		\[T(x,y,z)=(2x,y,z) \text{ for all } (x,y,z)\in \mathbb{X}.\]
		Suppose $u=(1,0,0)$ and $f\in \mathbb{X}^*$ given by $f(x,y,z)=x.$ Then $f \in J(u)$ and $T(\ker f)=\ker f.$ Thus, $T$ preserves Birkhoff-James orthogonality at $u$ with respect to $\ker f$ and therefore,  from  the above proposition, it follows that $u$ is a level vector of $T.$ Let $v=(\frac12, \frac12,0).$ Then $u\perp_B v$ but $Tu\not \perp_B Tv,$ as $\|Tu-2Tv\|=\|(0,-1,0)\|=1<2=\|Tu\|.$
	\end{example}
	Motivated by the characterization of norm attainment sets via semi-inner products in \cite[Th. 2.10]{SRT21}, we investigate a related problem in the setting of Birkhoff-James orthogonality. In this context, we obtain an analogous result by identifying the elements at which a bounded linear operator preserves Birkhoff-James orthogonality.
	
	\begin{proposition}\label{BJ for any}
		Let $T\in \mathbb{L}(\mathbb{X},\mathbb{Y}).$ Then $T$ preserves Birkhoff-James orthogonality at a nonzero $x\in \mathbb{X}$ if and only if for any semi-inner product $[\cdot, \cdot]_{\mathbb{X}_1}$ in $\mathbb{X},$ there exists a semi-inner product $[\cdot, \cdot]_{\mathbb{Y}_1}$ in $\mathbb{Y}$ such that
		\[[Ty,Tx]_{\mathbb{Y}_1}=\frac{\|Tx\|^2}{\|x\|^2}[y,x]_{\mathbb{X}_1} \text{ for all } y\in \mathbb{X}.\]
	\end{proposition}
	\begin{proof}
		Suppose that $T$ preserves Birkhoff-James orthogonality at $x(\neq 0)\in \mathbb{X}.$  Let $[\cdot, \cdot]_{\mathbb{X}_1}$ be a semi-inner product in $\mathbb{X}.$ Then there exists $f\in J(x)$ such that 
		\[[y,x]_{\mathbb{X}_1}=\|x\|f(y)~\text{ for all } y\in \mathbb{X}.\]
		Now, $\mathbb{X}=\spn \{x\} \oplus \ker f.$ Let $y\in \mathbb{X}$ be arbitrary.  Then there exist $h \in \ker f$ and $\alpha\in \mathbb{K}$ such that $y=\alpha x + h.$  Since, $T$ preserves Birkhoff-James orthogonality at $x$ and $x\perp_B h,$ it follows that $Tx\perp_B Th.$ Then there exists a semi-inner product $[\cdot, \cdot]_{\mathbb{Y}_1}$ in $\mathbb{Y}$ such that $[Th,Tx]_{\mathbb{Y}_1}=0.$
		Now,
		\begin{align*}
			[Ty,Tx]_{\mathbb{Y}_1}&= [T(\alpha x + h),Tx]_{\mathbb{Y}_1}\\
			&=\alpha \|Tx\|^2\\
			&= \frac{\|Tx\|^2}{\|x\|^2} ([\alpha x, x]_{\mathbb{X}_1}+[h,x]_{\mathbb{X}_1})\\
			&= \frac{\|Tx\|^2}{\|x\|^2} [y,x]_{\mathbb{X}_1}.
		\end{align*}
		Conversely, let for any semi-inner product $[\cdot, \cdot]_{\mathbb{X}_1}$ in $\mathbb{X},$ there exists a semi-inner product $[\cdot, \cdot]_{\mathbb{Y}_1}$ in $\mathbb{Y}$ such that
		\[[Ty,Tx]_{\mathbb{Y}_1}=\frac{\|Tx\|^2}{\|x\|^2}[y,x]_{\mathbb{X}_1}\text{ for all } y\in \mathbb{X}.\]
		Let $x\perp_Bz.$ Then there exists  a semi-inner product $[\cdot, \cdot]_{\mathbb{X}_2}$ in $\mathbb{X}$ such that $[z,x]_{\mathbb{X}_2}=0.$ By the hypothesis, there exists a semi-inner product $[\cdot, \cdot]_{\mathbb{Y}_2}$ in $\mathbb{Y}$ such that
		\[[Tz,Tx]_{\mathbb{Y}_2}=\frac{\|Tx\|^2}{\|x\|^2}[z,x]_{\mathbb{X}_2}=0.\] 
		This implies that $Tx \perp_B Tz.$ Thus, $T$ preserves Birkhoff-James orthogonality at $x.$
	\end{proof}

	The following result provides a necessary condition for a nonzero element of a  normed linear space to be a level vector of a bounded linear operator.
	\begin{proposition}
		If $x\in\mathbb{X}$ is a level vector of a  nonzero operator $T\in \mathbb{L}(\mathbb{X}, \mathbb{Y}),$  then either $Tx=0$ or $\ker T\subset x^{\perp_{B}}.$
	\end{proposition}
	
	\begin{proof}
		Let $x\in\mathbb{X}$ be a level vector of a  nonzero operator $T\in \mathbb{L}(\mathbb{X}, \mathbb{Y}).$ It follows from Proposition \ref{level char BJ} that  $T$ preserves Birkhoff-James orthogonality at $x$ with respect to $\ker f$ for some $f\in J(x).$ Suppose that $Tx\neq 0.$ If possible let $\ker T\not \subset x^{\perp_{B}}.$ Then there exists $y\in \ker T$ such that $y\notin x^{\perp_B}.$  So $\mathbb{X}=\spn\{y\}\oplus \ker f.$ Hence, there exist $\alpha(\neq 0)\in \mathbb{K}$ and $h(\neq 0)\in \ker f$ such that $x=\alpha y +h$ and so $Tx=Th.$ Now, $x\perp_B h\implies Tx\perp_B Th \implies Tx\perp_B Tx\implies Tx=0, $ a contradiction. Thus,  $\ker T\subset x^{\perp_{B}}.$ 
	\end{proof}
	
	Next, we obtain a geometric consequence of a bounded linear operator having a level vector corresponding to a nonzero level number on the relative interior of a face of the unit ball. To proceed, we require the following lemma.
	
	\begin{lemma}\label{int support}
		Let  $F$ be a face of $B_{\mathbb{X}}.$ Then for each $u\in \rint F,$ $J(u)\subset J(v)$ for all $v\in F.$
	\end{lemma}
	\begin{proof}
		Let $u\in \rint F.$ Then there exists	$\epsilon>0$ such that $\{y \in\aff(F):\|y-u\| \leq \epsilon\} \subset F.$ Let $f\in J(u).$ We show that $f\in J(v)$ for all $v\in F.$ Let $v\in F$ be arbitrary. Consider $\alpha \in \mathbb{R}$ such that  $1<\alpha \leq \frac{\epsilon+2}{2}.$ Then 
		\[\|u-((1-\alpha)v+\alpha u)\|\leq(\alpha-1)+(\alpha-1)=2\alpha-2\leq \epsilon.\]
		Therefore, $z=(1-\alpha)v+\alpha u\in F.$ Then $u=\big(1-\frac{1}{\alpha}\big)v+\frac{1}{\alpha}z.$ So,  $1=f(u)=\big(1-\frac{1}{\alpha}\big)f(v)+\frac{1}{\alpha}f(z)$ and this implies that $f(v)=1.$  Thus,  $f\in J(v)$ and this completes the proof.
	\end{proof}
	Now we have the desired result.
	\begin{theorem}\label{corres zero}	
		Let  $F$ be a face of $B_{\mathbb{X}}.$ If  there exists a  nonzero level number for an operator  $T\in\mathbb{L}(\mathbb{X},\mathbb{Y})$ corresponding  to a level vector  $u\in \rint F$   then  $Tv\neq 0$ for all $v\in F.$
	\end{theorem}
	\begin{proof}
		Let $k\neq 0$  be a   level number of $T\in\mathbb{L}(\mathbb{X},\mathbb{Y})$ corresponding  to a level vector  $u\in \rint F.$  If possible suppose that there exists  $v\in F$ such that $Tv=0.$ Now, there exist a semi-inner product $[\cdot, \cdot]_{\mathbb{X}}$ in $\mathbb{X}$ and a semi-inner product $[\cdot, \cdot]_{\mathbb{Y}}$ in $\mathbb{Y}$ such that 
		\[[Ty,Tu]_{\mathbb{Y}}=k[y,u]_{\mathbb{X}} \text{ for all } y\in \mathbb{X}.\]
		In particular for $y=v$, \[[Tv,Tu]_{\mathbb{Y}}=k[v,u]_{\mathbb{X}}\implies [v,u]_{\mathbb{X}}=0.\]
		Now, there exists  $f\in  J(u)$  such that 
		\[[y,u]_{\mathbb{X}}=f(y) \text{ for all } y\in \mathbb{X}.\]
		Since $u\in \rint F$ and $f\in J(u),$ it follows from Lemma \ref{int support} that $f\in J(y)$ for all $y\in F$ and so $f\in J(v).$ This implies that $1=f(v)=[v,u]_{\mathbb{X}},$ which contradicts the fact that $[v,u]_{\mathbb{X}}=0.$
		Therefore, $Tv\neq 0$ for all $v\in F.$
	\end{proof}
	As a direct implication of the above result, we have the following corollary.
	\begin{cor}\label{level0}
		Let $F$ be a face of $B_{\mathbb{X}}.$ Suppose that $T\in \mathbb{L}(\mathbb{X},\mathbb{Y}).$ If $Tv=0$ for some  $v\in F$ then there does not exist any level vector in  $\text{Int}_r F$ of $T$ with nonzero level number.
	\end{cor}
	
	Using Lemma \ref{int support} and Theorem \ref{corres zero}, we have the following observation.
	
	\begin{proposition}\label{equi}
		Let $T\in \mathbb{L}(\mathbb{X},\mathbb{Y})$ be nonzero and let $F$ be a face of $B_{\mathbb{X}}.$ If  $u,v\in \rint F$  are two level vectors of $T$ then the level numbers corresponding to the level vectors $u,v$ are same.
	\end{proposition}
	\begin{proof}
		If $Tu=0$ or $Tv=0,$ then it follows from Theorem \ref{corres zero} that $Tu=Tv=0.$
		Let $Tu\neq 0\neq Tv.$ Since $u$ is a  level vector of $T,$ it follows that there exist a semi-inner product $[\cdot, \cdot]_{\mathbb{X}_1}$ in $\mathbb{X}$ and a semi-inner product $[\cdot, \cdot]_{\mathbb{Y}_1}$ in $\mathbb{Y}$ such that 
		\[[Ty,Tu]_{\mathbb{Y}_1}=\|Tu\|^2[y,u]_{\mathbb{X}_1} \text{ for all } y\in \mathbb{X}.\] 
		Let  $[\cdot, \cdot]_{\mathbb{X}_1}$ be generated by $f\in J(u)$ at $u$ and  $[\cdot, \cdot]_{\mathbb{Y}_1}$ be generated by $g\in J(Tu)$ at $Tu.$ Thus,
		\[
		g(Ty)=\|Tu\|f(y) \text{ for all } y\in \mathbb{X}.
		\]
		In particular for $y=v$,  $g(Tv)=\|Tu\|f(v).$ Since $u,v\in \rint F,$  it follows from Lemma \ref{int support} that $f\in J(v)$ and so 
		$
		\|Tv\|\geq g(Tv)=\|Tu\|.
		$
		Similarly, as $v$ is a  level vector of $T,$ we have $\|Tu\|\geq\|Tv\|.$ Thus, $\|Tu\|=\|Tv\|.$
		This completes the proof.
	\end{proof}
	In the above result if any one of $u,v$ is not in the relative interior of $F,$ then the result may not hold. The following example illustrates the scenario.
	\begin{example}\label{exam not same}
		let $\mathbb{X}=\ell_{\infty}^2(\mathbb{R})$ and let the facet $F=\{(a,1):-1\leq a\leq 1\}$ of $B_{\mathbb{X}}.$ Consider the operator $T\in\mathbb{L}(\mathbb{X})$ defined by $T(x,y)=(2x,y)$ for all $(x,y)\in \mathbb{X}.$ Let $u=(0,1)$ and $v=(1,1).$ Clearly, $u,v\in F$ and both are level vector of $T.$ But the level number corresponding to $u$ is $1$ and the level number corresponding to $v$ is $4.$
	\end{example}
	In Proposition \ref{equi}, replacing the assumption that $T$ admits a level vector with the stronger condition that $T$ preserves Birkhoff-James orthogonality at $u$ and $v,$ we have the same conclusion for all $u,v \in F$. The analogous result in the setting of real Banach spaces was provided in \cite[Lemma 2.16]{MMPS025}. Since the proof in the general setting proceeds analogously, we state the lemma only.
	\begin{lemma}\label{equi}
		Let $T\in \mathbb{L}(\mathbb{X}, \mathbb{Y})$ preserve Birkhoff-James orthogonality at $u,v\in S_{\mathbb{X}}$. If $u,v$ are on the same face $F$ of $B_{\mathbb{X}}$ then $\|Tu\|=\|Tv\|.$
	\end{lemma}

	We now state the  following characterization of subspaces contained in the orthogonal region of a point in a normed linear space. The real case was proved in \cite[Lemma 2.14]{MMPS25}, and the same  argument applies to complex spaces, so we omit the proof.
	
	\begin{lemma}\label{subspace}
		Let  $x\in\mathbb{X}$ be nonzero and let $\mathbb{V}$  be a subspace of $\mathbb{X}.$ Then $\mathbb{V}\subset x^{\perp_B}$ if and only if there exists $f\in J(x)$ such that $\mathbb{V}\subset \ker f.$
	\end{lemma}
	
	Next, we explore some other geometric consequences of  level vectors of a bounded linear operator between normed linear spaces.
	\begin{theorem}\label{int face to face}
		Let $F$ be a face of $B_{\mathbb{X}}.$ Suppose that $T\in \mathbb{L}(\mathbb{X},\mathbb{Y}).$ If  each element of  $A\subset \rint F$ is a level vector of $T$ then the following are true:
		\begin{itemize}
			\item[(i)] If $T$ preserves  Birkhoff-James orthogonality  at  a point $x\in A$ with respect to $\ker f$ for some $f\in J(x)$ then $T$ preserves the same at each point of $A$ with respect to $\ker f.$  Moreover, there exists $g\in \bigcap\limits_{a\in A}J(Ta)$ such that $T(\ker f)\subset \ker g,$ unless $T(A)=\{0\}.$
			\item[(ii)] There exists a face $G$ of $B_{\mathbb{Y}}$ such that $\frac{Tu}{\|Tu\|}\in G$ for all $u\in A,$ unless $T(A)=\{0\}.$
			\item[(iii)] Each element of  $\co (A)$ is a level vector  of $T.$
		\end{itemize}
	\end{theorem}
	\begin{proof}
		(i) Let  each element of  $A\subset \rint F$ is a level vector of $T.$ From Proposition \ref {equi}, it follows that $\|Ta_1\|=\|Ta_2\|$ for all $a_1,a_2\in A.$ Let $T(A)\neq \{0\}$ and $x\in A.$ Suppose that $T$ preserves Birkhoff-James orthogonality at $x$ with respect to $\ker f,$ where $f\in J(x).$  Clearly, $f\in \bigcap\limits_{a\in A}J(a).$  Now, it follows from Lemma \ref{subspace} that there exists $g\in J(Tx)$ such that  $T(\ker f)\subset \ker g.$ We show that $g\in J(Ty)$ for all $y\in A.$ Let $z\in A.$  Now, $\mathbb{X}=\spn \{x\}\oplus\ker f$ and so $z=\alpha x+ h$ for some $\alpha\in \mathbb{K}$ and $h\in \ker f.$ Next, $1=f(z)=f(\alpha x+ h)=\alpha$ and so $z=x+h.$ Then 
		\[g(Tz)=g(Tx+Th)=g(Tx)=\|Tx\|=\|Tz\|.\]
		Therefore, $g\in J(Tz)$ and this completes the proof.\\
		
		(ii) Let $T(A)\neq \{0\}.$ From (i), it follows that  there exist $g\in J(Tu)$ for all $u\in A.$ Let $G$ be the face of $B_{\mathbb{Y}}$ supported by $g.$ Therefore,  $\frac{Tu}{\|Tu\|}\in G$ for all $u\in A.$\\
		
		(iii) If $T(A)= \{0\}$ then it is obvious. Suppose $T(A)\neq \{0\}.$   Let  $u\in \co(A).$ Then there exist  $f\in \bigcap\limits_{a\in A}J(a)$ and $g\in \bigcap\limits_{a\in A}J(Ta)$ such that $T(\ker f)\subset \ker g.$ Clearly, $f\in J(u).$ If we show that $g\in J(Tu)$ then we are done.  From Proposition \ref {equi}, it follows that $\|Tx\|=k$(say) for all $x\in A.$  Next, it follows from (ii) that there exists a face $G$ of $B_{\mathbb{Y}}$ such that $\frac{1}{k} Tx\in G$ for all $x\in A.$ Now, $u\in \co(A)\implies Tu\in \co(T(A))\implies \frac{1}{k}Tu\in \co(\frac{1}{k}T(A)).$ So $\frac{1}{k}Tu\in G.$ Since $g\in J(Tx)$ for all $x\in A,$ it follows that  that $g\in J(Tu).$
	\end{proof}

	\begin{remark}
		In the above theorem if  $A$ is not contained in the interior of $F,$ then the theorem may not hold. If we consider  Example \ref{exam not same}, then  $\frac{Tu}{\|Tu\|}, \frac{Tv}{\|Tv\|}$ are not in the same face of $B_{\mathbb{X}}.$ Furthermore, for  $t=\frac{2}{3}\in [0,1],$ $w=(1-t)u+tv=\big(\frac{2}{3},1\big)$ is not a level vector of $T.$
	\end{remark}
	
	In Theorem \ref{int face to face}, if we replace the assumption that $T$ admits a level vector by the condition that $T$ preserves Birkhoff-James orthogonality at each point of $A$, then the conclusion continues to hold for every $A \subset F$. The analogous result for real Banach spaces was established in \cite[Th.~2.17]{MMPS025}. The proof for general Banach spaces follows along the same line of arguments, and therefore we simply state the theorem.
	
	\begin{theorem}\label{faceToface}
		Let  $F$ be a face of $B_{\mathbb{X}}.$  If  $T\in \mathbb{L}(\mathbb{X},\mathbb{Y})$ preserves Birkhoff-James orthogonality on a set $A\subset F$ then the following results hold:
		\begin{itemize}
			\item[(i)] There exists a face $G$ of $B_{\mathbb{Y}}$ such that $\frac{Tu}{\|Tu\|}\in G$ for all $u\in A,$ unless $T(A)=\{0\}.$
			\item[(ii)] $T$ preserves Birkhoff-James orthogonality on $\co(A).$
		\end{itemize}
	\end{theorem}

	Let $L(T)$ be the set of all level numbers  of a nonzero operator $T$ defined on a $n$-dimensional real polyhedral Banach space. In \cite[Th. 2.6]{RST24}, it was proved that for a $n$-dimensional real polyhedral Banach space
	\[|L(T)|\leq \begin{cases}
		\frac{1}{2}\sum_{k=0}^{n-1}|F_k|+1,~ &\text{ if }  T \text{ is not injective},\\
		\frac{1}{2}\sum_{k=0}^{n-1}|F_k|,~ &\text{ if } T \text{ is injective}.
	\end{cases}
	\] Using Corollary \ref{level0} and Proposition  \ref{equi}, we obtain a  refinement of this bound.
	\begin{theorem}
		Let $\mathbb{X}$ be an $n$-dimensional real polyhedral Banach space. For any  $T\in \mathbb {L}(\mathbb{X}),$  
		\[|L(T)|\leq \begin{cases}
			\frac{1}{2}\left(\sum_{k=0}^{n-1}|F_k|-\sum_{k=0}^{m-1}|G_k|\right)+1,~ &\text{ if } \dim(\ker T)=m>0,\\
			\frac{1}{2}\sum_{k=0}^{n-1}|F_k|,~ &\text{ if } \dim (\ker T)=0,
		\end{cases}
		\]
		where $F_k$ is the collection of all $k$-faces of $B_{\mathbb{X}}$ and $G_k$ is the collection of all $k$-faces of $B_{\ker T}.$
	\end{theorem}
	The famous Kalai's $3^d$ conjecture \cite{K89} states that every centrally-symmetric $d$-polytope has at least $3^d-1$ proper faces. Recently, Sanyal et. al. in \cite{SW24} proved that any $d$-polytope, which is unconditional, i.e., invariant with respect to reflection in any  co-ordinate hyperplane, has at least $3^d-1$ proper faces.  We now obtain the following corollary. 
	\begin{cor}
		Let $\mathbb{X}$ be an $n$-dimensional real polyhedral Banach space such that $B_{\mathbb{X}}$ is unconditional.  For any  $T\in \mathbb {L}(\mathbb{X}),$  
		\[|L(T)|\leq \begin{cases}
			\frac{1}{2}\left(\sum_{k=0}^{n-1}|F_k|-3^m+1\right)+1,~ &\text{ if } \dim(\ker T)=m>0,\\
			\frac{1}{2}\sum_{k=0}^{n-1}|F_k|,~ &\text{ if } \dim (\ker T)=0,
		\end{cases}
		\]
		where $F_k$ is the collection of all $k$-faces of $B_{\mathbb{X}}.$
	\end{cor}
	
	In \cite[Th. 2.20]{SRT21}, the authors obtained a relation between the level vector of an operator defined on a reflexive, smooth, strictly convex Banach space and the level vector of its adjoint. We observe that the smoothness and strict convexity assumptions on the space  can be relaxed without any essential change to the argument presented there. To this end, we first prove the following lemma.
	
	\begin{lemma}\label{cano map}
		Let $\mathbb{X}$ be a reflexive  Banach space and let $x\in S_{\mathbb{X}}.$ Then for each $f\in J(x)$ there exists a semi-inner product $[\cdot, \cdot]_{\mathbb{X}^*}$ in $\mathbb{X}^*,$ such that
		\[[g,f]_{\mathbb{X}^*}=g(x) \text{ for all } g\in \mathbb{X}^*.\]
	\end{lemma}
	\begin{proof}
		Since $\mathbb{X}$ is reflexive, the canonical mapping $\phi:\mathbb{X}\longrightarrow \mathbb{X}^{**}$  is a norm preserving isomorphism. Let $\phi(y)=\psi_y$ for all  $y\in \mathbb{X},$ where $\psi_y(h)=h(y)$ for all  $h\in \mathbb{X}^*.$ Let $f\in J(x).$ We claim that $\psi_x\in J(f).$ Now, $\psi_x(f)=f(x)=1.$ Since $\phi$ is norm preserving, it follows that $\|\psi_x\|=\|x\|=1.$ Thus, our claim is established that $\psi_x\in J(f) .$ Then there exists  a semi-inner product $[\cdot, \cdot]_{\mathbb{X}^*}$ in $\mathbb{X}^*$ such that
		\begin{align*}
			[g,f]_{\mathbb{X}^*}&= \psi_x(f)\text{ for all } g\in \mathbb{X}^*\\
			&=g(x) \text{ for all } g\in \mathbb{X}^*.
		\end{align*}
	\end{proof}
	Using the above lemma, we now obtain our desired result.
	\begin{theorem}
		Let $\mathbb{X}$ be a reflexive Banach space and $\mathbb{Y}$ be any Banach space. If for any operator $T\in \mathbb{L}(\mathbb{X},\mathbb{Y}),$ $x\in S_{\mathbb{X}}\setminus \ker T$ is a level vector of $T$ then there exists $\psi\in J\Big( \frac{Tx}{\|Tx\|}\Big)$ such that $\psi$ is a level vector of $T^{\times}.$ Moreover, the level number of $T^{\times}$ corresponding to $\psi$ is equal to the level number of $T$ corresponding to $x.$
	\end{theorem}
	\begin{proof}
		Let $T\in \mathbb{L}(\mathbb{X},\mathbb{Y})$ and let  $x\in S_{\mathbb{X}}\setminus \ker T$ is a level vector of $T.$ Then there exist a semi-inner product $[\cdot, \cdot]_{\mathbb{X}}$ in $\mathbb{X}$ and a semi-inner product $[\cdot, \cdot]_{\mathbb{Y}}$ in $\mathbb{Y}$ such that 
		\[[Ty,Tx]_{\mathbb{Y}}=\|Tx\|^2[y,x]_{\mathbb{X}} \text{ for all } y\in \mathbb{X}.\]
		Now, for all $y\in \mathbb{X},$
		\begin{align*}
			\|y\|= \|x\|\|y\|
			&\geq|[y,x]_{\mathbb{X}}|\\
			&=\Big| \frac{1}{\|Tx\|^2}[Ty,Tx]_{\mathbb{Y}}\Big|\\
			&= \Big| \frac{1}{\|Tx\|}\Big[Ty,\frac{Tx}{\|Tx\|}\Big]_{\mathbb{Y}}\Big|\\
			&=\Big| \frac{1}{\|Tx\|}\psi(Ty)\Big|~\text{ for some } \psi\in J\Big( \frac{Tx}{\|Tx\|}\Big)\\
			&= \Big| T^{\times}\Big(\frac{1}{\|Tx\|}\psi\Big) (y)\Big|.\\
		\end{align*}
		Again, 
		\[ T^{\times}\Big(\frac{1}{\|Tx\|}\psi\Big) (x)= \psi\Big( \frac{Tx}{\|Tx\|}\Big)=1.\]
		Therefore, $T^{\times}\Big(\frac{1}{\|Tx\|}\psi\Big)\in J(x).$ From Lemma \ref{cano map}, it follows that  there exists a semi-inner product $[\cdot, \cdot]_{\mathbb{X}^*}$ in $\mathbb{X}^*$ such that 
		\[[g,\psi]_{\mathbb{X}^*}= g\Big(\frac{Tx}{\|Tx\|}\Big)=\frac{1}{\|Tx\|}g(Tx)=\frac{1}{\|Tx\|}(T^{\times}g)(x)~\text{ for all }g\in \mathbb{X}^*.\]
		Again, it follows from Lemma \ref{cano map} that there exists a semi-inner product $[\cdot, \cdot]_{\mathbb{Y}^*}$ in $\mathbb{Y}^*$ such that
		\[\Big[T^{\times}g,T^{\times}\Big(\frac{1}{\|Tx\|}\psi\Big)\Big]_{\mathbb{Y}^*}=(T^{\times}g)(x).\]
		This implies that 
		\[[T^{\times}g,T^{\times}\psi]_{\mathbb{Y}^*}=\|Tx\|(T^{\times}g)(x)=\|Tx\|^2[g,\psi]_{\mathbb{X}^*}.\]
		Therefore, $\psi$ is a level vector of $T^{\times}.$ Clearly,  the level number of $T^{\times}$ corresponding to $\psi$ is $\|Tx\|^2,$ which is equal to the level number of $T$ corresponding to $x.$
	\end{proof}
	The following corollary is an easy consequence of the above theorem and Proposition \ref{level char BJ}.
	\begin{cor}
		Let $\mathbb{X}$ be a reflexive Banach space and $\mathbb{Y}$ be any Banach space. If  any operator $T\in \mathbb{L}(\mathbb{X},\mathbb{Y})$ preserves Birkhoff-James orthogonality at $x\in S_{\mathbb{X}}\setminus \ker T$ with respect to $\ker f$ for some $f\in J(x)$ then there exists $\psi\in J\Big( \frac{Tx}{\|Tx\|}\Big)$ such that  $T^{\times}$ preserves Birkhoff-James orthogonality at $\psi$ with respect to $\ker \phi$ for some $\phi\in J(\psi).$
	\end{cor}
	
	\begin{remark}
		If   $T\in \mathbb{L}(\mathbb{X},\mathbb{Y})$ preserves Birkhoff-James orthogonality at $x\in S_{\mathbb{X}}\setminus \ker T$ then $T^{\times}$ may not preserve Birkhoff-James orthogonality at any   $\psi\in J\Big( \frac{Tx}{\|Tx\|}\Big).$  For example, consider $\mathbb{X}=\mathbb{Y}=\ell_{\infty}^3(\mathbb R)$ and $T\in \mathbb{L}(\mathbb{X})$ given by $T(x,y,z)=(x,2y,3z)$ for all $(x,y,z)\in \mathbb{X}.$ Let $u=(1,0,0)\in \mathbb{X}.$ Clearly, $J(T(u))=\{(1,0,0)\}\subset \ell_1^3(\mathbb R).$ Here $T$ preserves Birkhoff-James orthogonality at $u,$ but $T^{\times}$ does not preserve Birkhoff-James orthogonality at $(1,0,0).$
	\end{remark}

	The following result gives a necessary condition for $T\in \mathbb{L}(\mathbb X,\mathbb{Y})$ to admit every $x\in S_{\mathbb X}$ as its level vector.
	
	\begin{proposition}\label{injective}
		Let  $T \in \mathbb{L}(\mathbb{X},\mathbb{Y})$ be nonzero.  If each $x \in S_{\mathbb{X}}$ is a level vector of $T$ then $T$ is one-to-one.
	\end{proposition}
	\begin{proof}
		Let each $x \in S_{\mathbb{X}}$ be a level vector of $T.$ If possible suppose that $T$ is not one-to-one. Let $u \in \ker T$ and $v\in \mathbb{X}\setminus \ker T$ be such that $\|u\|=\|v\|=1.$ Suppose that $\mathbb{V}=span\{u,v\}.$ Choose  $y(\neq u) \in S_{\mathbb{V}}$ such that $\|y-u\|<1.$ Clearly, $y\not\perp_{B}u$ and $y\notin \ker T.$ Since $y$ is a level vector of $T,$ it follows that there exist a semi-inner product $[\cdot, \cdot]_{\mathbb{X}_1}$ in $\mathbb{X}$ and a semi-inner product $[\cdot, \cdot]_{\mathbb{Y}_1}$ in $\mathbb{Y}$ such that 
		\[[Tz,Ty]_{\mathbb{Y}_1}=\|Ty\|^2[z,y]_{\mathbb{X}_1} \text{ for all } z\in \mathbb{X}.\]
		In particular for $z=u,$
		\[[u,y]_{\mathbb{X}_1}=\frac{1}{\|Ty\|^2}[Tu,Ty]_{\mathbb{Y}_1}=0.\]
		This implies that $y\perp_{B}u,$ a contradiction. Therefore, $T$ is one-to-one.
	\end{proof}

	The next result presents a topological structure of the set of level vectors of a bounded linear operator.
	\begin{theorem}\label{closed} Let $T\in \mathbb{L}(\mathbb {X}, \mathbb{Y}).$ Then the set of all level vectors of  $T$ is closed in $\mathbb{X}\setminus\{0\}.$
	\end{theorem}
	\begin{proof} Let $L$ be the  set of all level vectors of $T.$ Consider $x(\neq 0)\in \overline{L}.$ Then there exists a sequence $\{x_n\}\subset \mathbb{X}\setminus\{0\}$ in $L$ such that $x_n\longrightarrow x.$ Since for each $n\in \mathbb{N},$ $x_n$ is a level vector of $T,$  there exists $f_n\in J(x_n) $ such that $T$ preserves Birkhoff-James orthogonality at $x_n$ with respect to $\ker f_n.$  Since $B_{\mathbb{X}^*}$ is weak* compact, there exists a weak* convergent net $\{f_{\alpha}\}_{\alpha\in A}$ of $\{f_n\},$ where $A$ is a directed set. Suppose that $f_{\alpha}\overset{w*}{\longrightarrow} f\in B_{\mathbb{X}^*}.$ Consider the subnet $\{x_{\alpha}\}_{\alpha\in A}$ of $\{x_n\},$  then $x_{\alpha}\longrightarrow x.$  Thus, 
		$f_{\alpha}(x_{\alpha})\longrightarrow f(x).$ 
		Since $f_{\alpha}(x_{\alpha})=\|x_{\alpha}\|$ for each $\alpha\in A,$ it follows that $f(x)=\|x\|.$ This implies that $f\in J(x).$ Now, we show that $T$ preserves Birkhoff-James orthogonality at $x$ with respect to $\ker f.$ Let $z\in \ker f.$ Next, we claim that there exist $z_{\alpha}\in \ker f_{\alpha}$ for all $\alpha\in A$ such that $z_{\alpha}\longrightarrow z.$  Let \[y\in\mathbb{X}\setminus\Big (\Big(\bigcup\limits_{\alpha\in A} \ker f_{\alpha}\Big)\bigcup \ker f\Big ).\]
		Then for each $\alpha\in A,$ there exists $r_{\alpha}\in \mathbb{K}$ such that $f_{\alpha}(z)=r_{\alpha}f_{\alpha}(y).$ Then \[r_{\alpha}f_{\alpha}(y)=f_{\alpha}(z)\longrightarrow f(z)= 0.\] Since $f_{\alpha}(y)\longrightarrow f(y)\neq 0,$ it follows that $r_{\alpha}\longrightarrow 0.$ For each $\alpha\in A,$ consider $z_{\alpha}=z-r_{\alpha} y.$  Then for each  $\alpha \in A,$ $z_{\alpha}\in \ker f_{\alpha}$ and $z_{\alpha}\longrightarrow z.$ Thus, our claim is established. 
		Now, $Tx_{\alpha}\perp_{B}Tz_{\alpha}$ for all $\alpha\in A.$ Then for each $\alpha\in A,$
		\[ \|Tx_{\alpha}+\lambda Tz_{\alpha}\|\geq \|Tx_{\alpha}\|\text{ for all }\lambda\in \mathbb{K}.\] This implies that
		\[ \|Tx+\lambda Tz\|\geq \|Tx\|\text{ for all }\lambda\in \mathbb{K}.\]
		So $Tx\perp_{B} Tz.$ Therefore, $T$ preserves Birkhoff-James orthogonality at $x$ with respect to $\ker f$ and so it follows from Proposition \ref{level char BJ} that $x$ is a level vector of $T.$ Thus,  $x\in L$ and therefore, $L$ is closed in $\mathbb{X}\setminus\{0\}.$
	\end{proof}
	
	From the above result, it follows that the set of all norm one level vectors of any operator $T$ on a normed linear space is closed. However, this set need not be compact, even when $T$ is a compact operator on a reflexive normed linear space. The following example illustrates this scenario.
	\begin{example}
		Let $\mathbb{X} = \ell^2(\mathbb R)$ and define an operator 
		$T : \mathbb{X} \to \mathbb{X}$ by 
		\[
		Te_n = \frac{1}{n} e_n \quad \text{for all } n \in \mathbb{N},
		\]
		where $\{e_n\}$ denotes the standard orthonormal basis of $\mathbb{X}$. 
		Clearly, $T$ is a compact operator. 
		It is easy to verify that $T$ preserves Birkhoff-James orthogonality at each $e_n$. 
		Hence, every $e_n$ is a  level vector of $T$. 
		However, the sequence $\{e_n\}$ admits no convergent subsequence in $S_{\mathbb{X}}$. Thus, the set of all norm one level vectors of $T$ is not compact.
	\end{example}

	We are now ready to characterize isometries up to scalar multiplication.
	\begin{theorem}\label{iso}
		Let  $T\in \mathbb{L}(\mathbb{X},\mathbb{Y}).$ Then $T$ is a scalar multiple of an isometry  if and only if each nonzero $x\in \mathbb{X}$ is a level vector. 
	\end{theorem}
	\begin{proof} The necessary part of the theorem is trivial. 
		We complete the proof of the theorem by proving the sufficient part in two steps. In  \textbf{Step 1}, we prove the theorem for  $2$-dimensional Banach spaces $\mathbb{X},\mathbb{Y}.$ Next, in  \textbf{Step 2}, we complete the proof using  \textbf{Step 1}.\\
		\textbf{Step 1}: Let $\mathbb{X},\mathbb{Y}$ be $2$-dimensional Banach spaces and let $T\in \mathbb{L}(\mathbb{X},\mathbb{Y}).$
		Suppose that each $x\in \mathbb{X}\setminus \{0\}$ is a level vector of $T.$
		If $T=0$ then the theorem holds trivially. Suppose $T\neq 0.$ From Lemma \ref{injective}, it follows that $T$ is bijective. Let $u\in \sm \mathbb{X}\cap T^{-1}(\sm \mathbb{Y})$ and $v\in \mathbb{X}$ be arbitrary.  Suppose that $J(u)=\{f\}$ and $J(Tu)=\{g\}.$ Let $[\cdot, \cdot]_f$   be the semi-inner product in $\mathbb{X}$  generated by $f$ at $u$  and $[\cdot, \cdot]_g$ be the semi-inner product in $\mathbb{Y}$ generated by $g$ at $Tu.$  Since $u$ is a level vector of $T,$ it follows that 
		\begin{eqnarray*}&&[Tv,Tu]_g=\frac{\|Tu\|^2}{\|u\|^2}[v,u]_f\\&\implies& g(Tv)=\frac{\|Tu\|}{\|u\|} f(v).\end{eqnarray*} Therefore, for any $x \in \sm \mathbb{X}\cap T^{-1}(\sm \mathbb{Y})$ and $y\in \mathbb{X},$ \[g_{Tx}(Ty)=\frac{\|Tx\|}{\|x\|} f_x(y),\]where $f_x\in J(x)$ and $g_{Tx}\in J(Tx).$
		Let us consider the function $\phi:\mathbb{X}\setminus\{0\}\longrightarrow \mathbb{R}$ defined by
		$\phi(x)=\frac{\|Tx\|}{\|x\|}.$ Clearly, $f$ is a continuous real-valued function on $\mathbb{X}\setminus\{0\}.$ Let $x\in \sm \mathbb{X}\cap T^{-1}(\sm \mathbb{Y}).$ Then for any $y\in \mathbb{X},$	
		\begin{eqnarray*}
			&&\lim\limits_{t\to 0}\frac{\phi(x+ty)-\phi(x)}{t}\\
			&&~=~\lim\limits_{t\to 0}\frac{1}{t}\Bigg(\frac{\|Tx+tTy\|}{\|x+ty\|}-\frac{\|Tx\|}{\|x\|}\Bigg)\\
			&&~=~\lim\limits_{t\to 0}\frac{1}{t}\Bigg\{\frac{\|x\|\big(\|Tx+tTy\|-\|Tx\|\big)-\|Tx\|\big(\|x+ty\|-\|x\|\big)}{\|x+ty\|\|x\|}\Bigg\}\\
			&&~=~\frac{\|Tx\|}{\|x\|}\Bigg\{\frac{1}{\|Tx\|}\lim\limits_{t\to 0}\frac{\|Tx+tTy\|-\|Tx\|}{t}-\frac{1}{\|x\|}\lim\limits_{t\to 0}\frac{\|x+ty\|-\|x\|}{t}\Bigg\} 	\\
			&&~=~\frac{\|Tx\|}{\|x\|}\Bigg\{\frac{1}{\|Tx\|}g_{Tx}(Ty)-\frac{1}{\|x\|}f_x(y)\Bigg\},	\text{ where $f_x\in J(x)$ and $g_{Tx}\in J(Tx).$}\\
			&&~=~0.
		\end{eqnarray*}
		This shows that $\phi$ is constant on $\sm \mathbb{X}\cap T^{-1}(\sm \mathbb{Y}).$  
		Since $T$ is bijective and $\mathbb {X}$ is $2$-dimensional, it follows  that $\sm \mathbb{X}\cap T^{-1}(\sm \mathbb{Y})$ is dense in $\mathbb{X}.$ As $\phi$ is a real-valued continuous function on $\mathbb{X}\setminus\{0\},$ we conclude that $\phi$ is constant on $\mathbb{X}\setminus\{0\}.$ Therefore, $\frac{\|Tx\|}{\|x\|}=\phi(x)=\lambda$ (say) for all $x\in \mathbb{X}\setminus\{0\}.$ This implies $ \|Tx\|=\lambda\|x\|$ for all $x\in \mathbb{X}.$ \\
		\textbf{Step 2}: Let $\mathbb{X},\mathbb{Y}$ be any normed linear spaces. Let each $x\in \mathbb{X}\setminus\{0\}$ is a level vector of $T.$  If $T=0$ then the theorem holds trivially. Suppose $T\neq 0.$ From Lemma \ref{injective}, it follows that $T$ is injective. Let $u,v\in\mathbb{X}\setminus\{0\}$ be arbitrary. We show that $\frac{\|Tu\|}{\|u\|}=\frac{\|Tv\|}{\|v\|}.$ If $u,v$ are linearly dependent, then it is obvious. Let $u,v$ are linearly  independent. Consider the subspaces $\mathbb{U}=\spn~\{u,v\}$ and $\mathbb{V}=\spn~\{Tu,Tv\}.$ Clearly, $\mathbb{U}, \mathbb{V}$ both are two-dimensional subspaces of $\mathbb{X}.$ Then each element of $\mathbb{U}$ is a level vector of the restricted operator $T|_{\mathbb{U}}:\mathbb{U}\longrightarrow\mathbb{V}.$ From \textbf{Step 1}, it follows that $\frac{\|T|_{\mathbb{U}}u\|}{\|u\|}=\frac{\|T|_{\mathbb{U}}v\|}{\|v\|}.$ Therefore, $\frac{\|Tu\|}{\|u\|}=\frac{\|Tv\|}{\|v\|}.$ This completes the proof.
	\end{proof}
	
	\begin{remark}
		\begin{itemize}
			\item[(i)]  It is well known that if every element of an inner product space is a right singular vector of a bounded linear operator, then the operator must be a scalar multiple of an isometry. Since the notion of level vectors naturally extends the notion of right singular vectors to arbitrary normed linear spaces, the preceding result provides a corresponding generalization to the setting of normed linear spaces.
			
			\item[(ii)] From Proposition \ref{level char BJ}, it follows that if a bounded linear operator preserves Birkhoff-James orthogonality at a point with respect to the kernel of a supporting functional at that point, then the operator admits that point as its level vector. Thus, the above characterization of isometries  provides  a proper refinement of the  Blanco-Koldobsky-Turn\v{s}ek characterization of  isometries.
		\end{itemize}
	\end{remark}
	
	The following corollary provides an additional refinement of the Blanco-Koldobsky-Turn\v{s}ek characterization  of  isometries and also improves the result \cite[Th. 2.9 ]{MMPS025}. The proof follows directly from Theorem \ref{closed} and Theorem \ref{iso}.
	\begin{cor}\label{density}
		Let $T\in \mathbb{L}(\mathbb{X},\mathbb{Y})$ and $D$ be a dense subset  of $\mathbb{X}.$ Then the following are equivalent:
		\begin{itemize}
			\item[(i)] Each nonzero element of $D$ is a level vector of  $T.$ 
			\item[(ii)] $T$ preserves Birkhoff-James orthogonality at each $x\in D$ with respect to $\ker f$ for some $f\in J(x).$
			\item[(iii)] $T$ is a scalar multiple of an isometry.
		\end{itemize}
	\end{cor}
	
	In \cite[Th. 2.19 ]{MMPS025}, it was established that a norm one operator on a real Banach space is an isometry if it preserves Birkhoff-James orthogonality at all extreme points of the unit ball, provided the unit ball satisfies the \textit{Krein-Milman property} and the set of smooth points forms a dense $G_{\delta}$ subset. We note that, in arbitrary normed linear spaces, the same conclusion holds even without the latter assumption, and the core argument remains unchanged.
	\begin{theorem}\label{extreme}
		Let $\mathbb{X}$ be a  normed linear space such that $B_{\mathbb{X}}=\overline{\co(\ext B_{\mathbb{X}})}$ and let $T\in S_{\mathbb{L}(\mathbb{X})}.$  Then $T$ is an isometry if and only if $T$ preserves Birkhoff-James orthogonality at each  $x \in \ext B_{\mathbb{X}}$.
	\end{theorem}
	\begin{proof}
		We prove the sufficient part only, as the necessary part is obvious. Let $T$ preserve Birkhoff-James orthogonality at each  point of $\ext 
		B_{\mathbb{X}}.$ By using Theorem \ref{faceToface}, it is easy to observe that $T$ preserves Birkhoff-James orthogonality at each  point of $\co(\ext B_{\mathbb{X}})$ and consequently, each  point of $\co(\ext B_{\mathbb{X}})$ is a level vector of $T.$  Now, it follows from  \ref{closed} that   each  point of $\overline{\co(\ext B_{\mathbb{X}})}=B_{\mathbb{X}}$ is a level vector of $T.$ Therefore, $T$ is an isometry.
	\end{proof}
	
	\begin{remark}
		\begin{itemize}
			
			\item[(i)] In Theorem \ref{extreme}, the assumption that $B_{\mathbb{X}}=\overline{\co(\ext B_{\mathbb{X}})}$ is essential. For instance, if we consider $\mathbb{X}=\ell_1\oplus_1 c_0$ then clearly $B_{\mathbb{X}}\neq\overline{\co(\ext B_{\mathbb{X}})}$ and the operator $T\in S_{\mathbb{L}(\mathbb{X})}$ defined by $T(x,y)=(0,y)$ for all $x\in \ell_1$ and $y\in c_0$   is not an isometry  but preserves Birkhoff-James orthogonality at each point of $\ext B_{\mathbb{X}}.$ 
			
			\item[(ii)] We note that for any finite-dimensional Banach space $\mathbb{X}$,  $B_{\mathbb{X}}=\overline{\co(\ext B_{\mathbb{X}})}.$ Lindenstrauss \cite[Th. 3.3.6]{B83} established in 1973 that if a closed convex subset $A$ of a Banach space $\mathbb{X}$ possesses the \textit{Radon-Nikod\'ym property} (RNP), then it also has the \textit{Krein-Milman property}, i.e., $A =\overline{\co(\ext A)}.$
			Therefore, every Banach space $\mathbb{X}$ with the RNP satisfies  $B_{\mathbb{X}}=\overline{\co(\ext B_{\mathbb{X}})}.$
		\end{itemize}
	\end{remark}
	
	Now in the following example we show that for a bounded linear operator on  a  normed linear space $\mathbb{X}$ with $B_{\mathbb{X}}=\overline{\co(\ext B_{\mathbb{X}})},$ if  each  point of $\ext 
	B_{\mathbb{X}}$ is a level vector then the operator may not be a scalar multiple of an ismetry.
	\begin{example}
		Let $\mathbb{X} = \ell_1^3(\mathbb{R})$ and consider the operator $T \in \mathbb{L}(\mathbb{X})$ defined by
		\[
		T(\alpha_1, \alpha_2, \alpha_3) = (\alpha_1, 2\alpha_2, 3\alpha_3) \text{ for all } (\alpha_1, \alpha_2, \alpha_3)\in \mathbb{X}.
		\]
		Clearly, $T$ is not a scalar multiple of an isometry. Now, $\ext B_{\mathbb{X}} = \{\pm u_i : i = 1,2,3\},$ where $u_1 = (1,0,0)$, $u_2 = (0,1,0)$ and $u_3 = (0,0,1).$ For each $i = 1,2,3$, let $f_i \in \mathbb{X}^*$ be defined by $f_i(\alpha_1, \alpha_2, \alpha_3) = \alpha_i$. 
		Then $f_i \in J(u_i)$ and $T(\ker f_i) = \ker f_i$ for each $i$. Consequently, $T$ preserves Birkhoff-James orthogonality at $u_i$ with respect to $\ker f_i$ for all $i = 1,2,3$. By Proposition \ref{level char BJ}, it follows that each $u_i$ is a level vector of $T$.
	\end{example}

	Next, we observe that  for a norm one, non-identity operator on an $n$-dimensional Banach space, there does not exist any set of $n$ linearly independent eigenvectors containing a smooth level vector such that any element of that set is not Birkhoff-James orthogonal to that point.
	\begin{proposition}
		Let $\mathbb{X}$ be an $n$-dimensional Banach space  and let $T\in\mathbb{L}(\mathbb{X})$ be nonzero. Then $T$ is a scalar multiple of the identity if and only if  there exist
		$n$ linearly independent $x_1, x_2,\ldots,x_n\in S_\mathbb{ X}$ such that 
		\begin{itemize}
			\item[(i)] $x_1, x_2,\ldots,x_n$ are eigenvector of $T,$
			\item[(ii)] $x_1\notin\ker T$ is smooth,
			\item[(iii)] $x_1$ is a level vector of $T,$ 
			\item[(iv)] $x_1\not\perp_{B} x_i$ for all $2\leq i\leq n.$ 
		\end{itemize}
	\end{proposition}
	\begin{proof}
		We only prove the sufficient part as the necessary part is obvious.
		Let $\lambda_1, \lambda_2,\ldots,$ $ \lambda_n$ are   eigenvalues of $T$ corresponding  to $x_1, x_2, \ldots ,x_n,$ respectively. Let $J(x)=\{f\}.$ Since $x_1\not\perp_{B} x_i$ for all $2\leq i\leq n,$ it follows that $f(x_i)=k_i\neq0$ for all $2\leq i\leq n.$ Then $x_1\perp_{B} k_ix_1-x_i$ for all $2\leq i\leq n.$
		Since $x_1$ is a level vector of $T,$ it follows form Proposition \ref{level char BJ} that $T$ preserves Birkhoff-James orthogonality at $x_1.$ So for all $2\leq i\leq n,$
		\begin{eqnarray*}
			Tx_1\perp_{B} k_iTx_1-Tx_i&\implies& x_1\perp_{B} k_i\lambda_1 x_1-\lambda_ix_i\\
			&\implies& f(k_i\lambda_1 x_1-\lambda_ix_i)=0\\
			&\implies& k_i\lambda_1-\lambda_ik_i=0\\
			&\implies& \lambda_1=\lambda_i.
		\end{eqnarray*}
		Thus, $\lambda_1=\lambda_2=\dots=\lambda_n=\lambda$(say). Then it is easy to observe that each nonzero point of $\mathbb{X}$ is a eigenvector of $T.$ Therefore, $T$ is a scalar multiple of the identity.
	\end{proof}
	\begin{remark}
		In the above theorem, the result may not hold if $x_1, x_2,\ldots,x_n\in S_\mathbb{ X}$ do not satisfy at least one of the four required conditions. For example consider the Banach space $\mathbb{X}=\ell_{\infty}^3(\mathbb{R})$ and let   $T,S\in \mathbb{L}(\mathbb{X})$ defined by $T(x,y,z)=(3x-2y,x,z)$ and $S(x,y,z)=(x,y,2z)$ for all $(x, y, z)\in \mathbb{X}.$ It is clear that neither $T$ nor $S$ is a scalar multiple of the identity operator, but the following still hold.
		\begin{itemize}
			\item[(1.)] Let $x_1=(1,0,0),x_2=(1,\frac{1}{2},0), x_3=(1,0,\frac{1}{2}).$ Clearly, $x_1,x_2,x_3\in S_{\mathbb{X}}$ are linearly independent and note that $x_3$ is not an  eigenvector of $S.$ Observe that $x_1$ is smooth and $x_1\not\perp_B x_i$  for $i=2,3.$ It can be easily verified that $S$ preserves Birkhoff-James orthogonality at $x_1$ with respect to $\ker f,$ where $f(x,y,z)=x$  and $f\in J(x_1).$ Consequently,  $x_1$ is a level vector of $S.$ Therefore, $x_1,x_2,x_3$ satisfy only the conditions (ii),(iii) and (iv).
			\item[(2.)] Let $x_1=(1,1,0),x_2=(1,\frac{1}{2},0), x_3=(1,1,\frac{1}{2}).$ Then $x_1,x_2,x_3\in S_{\mathbb{X}}$ are linearly independent  eigenvectors of $T.$ Observe that $x_1$ is not smooth and satisfies $x_1\not\perp_Bx_i$  for  $i=2,3.$ Now, $T$ preserves Birkhoff-James orthogonality at $x_1$ with respect to $\ker f,$ where $f(x,y,z)=x$  and $f\in J(x_1).$ Thus, $x_1$ is a level vector of $T.$ Therefore, $x_1,x_2,x_3$ satisfy only the conditions (i),(iii) and (iv).
			\item[(3.)]  Let $x_1=(1,\frac{1}{2},0),x_2=(1,1,0), x_3=(1,1,\frac{1}{2}).$ Then $x_1,x_2,x_3\in S_{\mathbb{X}}$ are linearly independent  eigenvectors of $T.$ Clearly, $x_1$ is smooth and $x_1\not\perp_Bx_i$  for  $i=2, 3.$ It is easy to observe that  $x_1$ is not a level vector of $T.$ Therefore, $x_1,x_2,x_3$ satisfy only conditions (i),(ii) and (iv).
			\item[(4.)]Let $x_1=(1,0,0),x_2=(1,\frac{1}{2},0), x_3=(0,0,1).$ Then $x_1,x_2,x_3\in S_{\mathbb{X}}$ are linearly independent   eigenvectors of $S.$ Clearly, $x_1$ is smooth and  a level vector of $S,$ but $x_1\perp_{B} x_3.$ Therefore, $x_1,x_2,x_3$ satisfy only conditions (i),(ii) and (iii).
			
		\end{itemize}
	\end{remark}

	We  observed that for a finite-dimensional Banach space $\mathbb X,$ if a linear operator preserves Birkhoff-James orthogonality on $\ext B_{\mathbb{X}}$ then the operator must be an isometry multiplied by a constant. The question that we raise here is that whether $\ext B_{\mathbb{X}}$ can be replaced by a proper subset of $\ext B_{\mathbb{X}}.$ A more specific question is the following.\\
	
	\textbf{Open Question:}
	\textit{In an $n$-dimensional polyhedral Banach space, does the preservation of Birkhoff-James orthogonality by a bounded linear operator at any $n$ linearly independent extreme points necessarily imply that the operator is a scalar multiple of an isometry?}\\
	
	In case of an arbitray Banach space the result is not necessarily correct, which follows from \cite[Th. 2.33]{MMPS25}.

	\section*{Acknowledgments}
	Jayanta Manna would like to thank UGC, Govt. of India, for the support
	in the form of Senior Research Fellowship under the mentorship of Professor Kallol Paul  and Dr Debmalya Sain.
	The research of Dr Kalidas Mandal and 
	Professor Kallol Paul is supported by CRG Project bearing File no. CRG/2023/00716 of DST-SERB, Govt.
	of India.
	
	\section*{Declarations}
	
	\begin{itemize}
		
		\item Conflict of interest
		
		The authors have no relevant financial or non-financial interests to disclose.
		
		\item Data availability 
		
		Data availability is not applicable to this article as no new data were created or analysed in this study.
		
		\item Author contribution
		
		All authors contributed to the study. All authors read and approved the final version of the manuscript.
		
	\end{itemize}

\end{document}